\documentclass[12pt, reqno]{amsart}
\usepackage[utf8]{inputenc}

\usepackage{amssymb}
\usepackage{graphics}
\usepackage{graphicx}
\usepackage{amscd}
\usepackage{color}
\usepackage{amsmath, amsthm, epsfig,  psfrag}
\usepackage{pinlabel}
\usepackage{hyperref}
\usepackage{subcaption}
\usepackage[margin=3cm, marginpar=2.5cm]{geometry}

\newtheorem{theorem}{Theorem}[section]
\newtheorem{lemma}[theorem]{Lemma}

\newtheorem{cor}[theorem]{Corollary}
\newtheorem{question}[theorem]{Question}

\theoremstyle{definition}
\newtheorem{definition}[theorem]{Definition}
\newtheorem{remark}[theorem]{Remark}
\newtheorem{example}[theorem]{Example}

\newcommand{\on}{\operatorname}
\renewcommand{\d}{\partial}
\renewcommand{\Im}{\on{Im}}

\newcommand{\T}{\mathcal{T}}
\newcommand{\V}{\mathcal{V}}

\newcommand{\eps}{\varepsilon}

\newcommand{\s}{\mathfrak{s}}
\newcommand{\ts}{\mathfrak{t}}

\newcommand{\CC}{\mathbb{C}}

\newcommand{\RR}{\mathbb{R}}
\newcommand{\ZZ}{\mathbb{Z}}
\newcommand{\QQ}{\mathbb{Q}}
\newcommand{\HH}{\mathbb{H}}

\newcommand{\FF}{\mathbb{F}}
\newcommand{\LL}{\mathbb{L}}

\newcommand{\CP}{{\mathbb {CP}}}

\newcommand{\ti}{\tilde}

\DeclareMathOperator{\Hom}{Hom}

\DeclareMathOperator{\PD}{PD}

\DeclareMathOperator{\Ker}{Ker}
\DeclareMathOperator{\Char}{Char}

\usepackage{relsize}

\title[Heegaard Floer contact invariants of links of singularities]{Heegaard
Floer invariants of contact structures on links of surface singularities}

\author{J\'ozsef Bodn\'ar}
\author{Olga Plamenevskaya}
\thanks{OP was partially supported by NSF grant DMS-1510091 and a Simons
Fellowship.}
\address{Department of Mathematics, Stony Brook University, Stony Brook, NY,
11794,  U.S.A.}
\email{olga@math.stonybrook.edu}
\email{jozef.bodnar@stonybrook.edu}

\begin{document}

\begin{abstract} Let a contact 3-manifold $(Y, \xi_0)$ be the 
link of a normal surface singularity
 equipped with its canonical contact structure $\xi_0$. 
 We prove a special property of such contact 3-manifolds of ``algebraic''
origin: the Heegaard Floer invariant $c^+(\xi_0)\in HF^+(-Y)$ cannot lie in the
image of the $U$-action on $HF^+(-Y)$.
It follows that Karakurt's ``height of $U$-tower'' invariants are always 0
for canonical contact structures on singularity links, which contrasts the fact
that the height of $U$-tower can be arbitrary for general fillable contact
structures. Our proof  uses the interplay between the Heegaard Floer homology
and N\'emethi's lattice cohomology.
\end{abstract}

 \maketitle

\section{Introduction and background}

Consider a complex surface $\Sigma\subset \CC^N$ with an isolated critical point
at the origin. 
For a sufficiently small $\eps>0$, the intersection $Y= \Sigma \cap
S^{2N-1}_{\eps}$ with the
sphere $S^{2N-1}_{\eps}= \{|z_1|^2+|z_2|^2+\dots +|z_N|^2 = \eps\}$ is a smooth
3-manifold called the {\em link of the singularity}.
The complex structure on $\Sigma$ induces the {\em canonical}
contact structure $\xi_0$ on $Y$ given by the distribution of complex
tangencies; in this setting, $(Y, \xi_0)$ is also called {\em Milnor fillable}.
The contact manifold $(Y, \xi_0)$ is independent of the choice of $\eps$, up to
contactomorphism. Moreover, it is shown in  \cite{CNP} that 
the Milnor fillable contact structure on  $Y$ is unique (note that in general, a link of singularity may support a number
of tight or Stein fillable contact structures).  The
Milnor fillable structure can be thought of as the contact structure
closely related to the algebraic origin of the manifold $Y$ as link of
singularity (and potentially carry information about the singularity). 
We would
like to address

\begin{question} Are there any special features that distinguish the canonical
contact structure from other contact structures on the link of
singularity?
\end{question}

It is known, for example, that $\xi_0$ is always Stein fillable \cite{BO} and
universally
tight~\cite{LO}.  In this paper, we
work with Ozsv\'ath-Szab\'o's Heegaard Floer homology \cite{OS} and
N\'emethi's lattice cohomology \cite{Ne1, Ne2} to establish  special
properties of the Heegaard Floer contact invariant $c^+$ (introduced in \cite{OScont}) of  canonical contact
structures. Recall that for a
$3$-manifold $Y$, the Heegaard Floer homology $HF^+(Y)$ is an 
$\FF[U]$-module (coefficients are assumed to be $\FF=\ZZ/2$, see
Remark~\ref{Zcoeffs} for $\ZZ$ coefficients). We review the context and
background after
stating our main result in terms of the $U$-action.

\begin{theorem} \label{main} Let $(Y, \xi_0)$ be a rational homology sphere link
of a normal surface
singularity
with its canonical contact structure, and  $c^+(\xi_0) \in HF^+(-Y)$ its
contact
invariant. Assume that the singularity is {\em not} rational.  Then $c^+(\xi)
\notin \Im U$.  
\end{theorem}

A singular point $p$ is {\em normal} when bounded holomorphic functions defined
in its
punctured neighborhood can be extended over $p$.  More importantly to us, 
normality together with the homological assumption on $Y$ is equivalent to
saying that $Y$ is the boundary of a negative-definite 4-manifold which is a
plumbing of spheres such that the plumbing graph is a tree (see
section~\ref{s:plumbings}).  

Given a 3-manifold $Y$, recall that its Heegaard Floer homology, developed 
in \cite{OS} and sequels, is an $\FF[U]$-module 
$HF^+(Y)$ that decomposes as a direct sum of components corresponding to
$Spin^c$ structures on $Y$.
When $Y$ is a rational homology sphere, $HF^+(Y, \s) = \T
\oplus \mathrm{Torsion}$, where $\T$ is 
isomorphic to $\FF[U,U^{-1}] / U \FF[U]$ (with the appropriate grading)
and $\mathrm{Torsion}$ is annihilated by $U^d$ for some large $d$.  A rational
homology sphere $Y$ is called an $L$-space when its
Heegaard Floer homology is the simplest possible, i.e.  $HF^+(Y, \s) = \T$ for
every $\s \in Spin^c(Y)$. In the case where $Y$ is the link of a normal surface
singularity, it is known that $Y$ is an $L$-space if and only if the
singularity is {\em rational}, \cite{OSplumb,NeL}. (We will not discuss
algebro-geometric definition of rational singularities here; in fact the reader
can take the $L$-space criterion above as a definition.)

Given a contact 3-manifold $(Y, \xi)$,  its
invariant $c^+(\xi)$ is defined as a
distingushed element of the Heegaard Floer group $HF^+(-Y)$, \cite{OScont}.
More precisely, $c^+(\xi) \in
HF^+(-Y, \ts_{\xi})$, where $\ts_{\xi}$ is the $Spin^c$ structure induced by
$\xi$.  For Stein fillable contact structures, the invariant is non-zero,
 in particular, 
$c^+(\xi_0) \neq 0$
for the canonical contact structure $\xi_0$ on a link
of any surface singularity. 
For an arbitrary contact 3-manifold $(Y, \xi)$, the contact invariant 
is annihilated by the $U$-action, i.e. $c^+(\xi) \in \Ker U$.

 The $\FF[U]$-module structure was used by
Karakurt in \cite{Ka} to define a related numerical invariant of contact
structures. More precisely, Karakurt considers the height of $U$-tower over
$c^+(\xi)$ to define 
$$
ht(\xi) = \max \{n: c^+(\xi) \in U^{n} \cdot HF^+(-Y) \}. 
$$
We have taken the liberty of changing the sign in Karakurt's original
definition; in \cite{Ka}, the invariant is defined as $\sigma(\xi) = -
ht(\xi)$. Karakurt computes $ht$ for a number of contact structures obtained
by Legendrian surgery, and shows that $ht$ can take arbitrary integer values
from $0$ to $+\infty$. In \cite{KO}, Karakurt and \"{O}zt\"{u}rk
show that the height of tower is 0 for canonical contact structures on links of
``almost rational'' (AR) singularities, using the fact that
Heegaard Floer homology is isomorphic to N\'{e}methi's lattice
cohomology \cite{Ne1, Ne2} for 3-manifolds of this type. For rational
singularities, it is easy to see that $ht=+\infty$ for every contact structure
on the link: this follows from the fact that the link $Y$ of a rational
singularity is an
$L$-space, i.e. $HF^+(-Y, \s) = \T$ for every $Spin^c$ structure $s$ on $Y$,
\cite{OSplumb,Ne1}. Karakurt-\"{O}zt\"{u}rk ask whether height of tower can
take arbitrary
integer values for canonical contact structures on links of general normal
surface singularities \cite[Question 6.2]{KO}. It follows immediately from
Theorem~\ref{main}  that the answer is manifestly no:

\begin{cor} Consider a normal surface singularity which is not rational and its link is a rational homology sphere. 
Let $\xi_0$ be the canonical contact structure on the link. Then $ht(\xi_0)=0$.
\end{cor}

When starting this work, our initial goal was to use the height of tower
invariants (together with their monotonicity under Stein cobordisms \cite{Ka})
to obstruct certain deformations of surface singularities.
The above corollary
means, however, that the $ht$ invariant contains very little information about
the given singularity! (One could use $ht$ to show that rational
singularities
cannot be deformed into non-rational, but this is a well-known fact and a
special case of the semicontinuity of the geometric genus, see \cite{Elk}. The algebro-geometric proof of this fact is non-trivial, 
so the Heegaard Floer argument may still be of interest.)

Similarly to \cite{KO}, our proof also uses the interplay between Heegaard
Floer homology and lattice cohomology of \cite{Ne1, Ne2}. Lattice cohomology is
defined in a combinatorial way, using the intersection lattice of the plumbing
graph (the dual resolution graph of the singularity). Under certain rather
restrictive conditions (for example, for links of  AR singularities), the
Heegaard Floer homology and lattice
cohomology are known to be isomorphic \cite{OSplumb, Ne2}. For arbitrary
3-manifolds,
a
spectral sequence
from lattice homology to Heegaard Floer homology was found in \cite{OSS-ss};
this spectral sequence collapses in certain special cases, but in general,
isomorphism between Heegaard Floer and lattice (co)homologies has not been
established. The isomorphism between the  Heegaard Floer and lattice
theories  in the case of AR singularities is the key tool in Karakurt's
and Karakurt-\"Ozt\"urk's  proofs in \cite{Ka, KO}. Our approach
is different in that we only use 
an $\FF[U]$-equivariant map from the Heegaard Floer homology
to the lattice cohomology and do not require an
isomorphism, thus our argument works in general.  The homomorphism we use comes
from \cite{OSplumb} and maps $HF^+(-Y)$ to the $0$-dimensional part 
$\HH_0^+(\Gamma)$ of lattice cohomology. 
(The latter is much simpler that the full lattice cohomology 
$\HH^{\ast}(\Gamma)$; note that for AR-singularities,
lattice cohomology vanishes in dimensions $n>0$, \cite{NemLasz}.)
Another
difference between
our work and \cite{KO} is that we use general properties of graded roots
without resorting to Laufer sequences specific to the AR case. Finally, although this is not stated explicitly, the proofs in \cite{KO} only work if
   the minimal resolution of the singularity has exceptional divisor with normal crossings 
(see discussion in Section~\ref{s:plumbings} and Section~\ref{s:blowups}). This condition 
often fails, even for AR singularities; we give a proof for the general case.

It is intriguing that the proof of Theorem~\ref{main} works with lattice
cohomology to establish a statement about Heegaard Floer invariants, even in
the absence of isomorphism between the two theories. It would be very
interesting to find further similar applications of lattice cohomology. 

The paper is organized as follows. In Sections~\ref{s:plumbings} and \ref{s:latticehom}, we collect 
some background facts on resolutions and describe different constructions of lattice 
cohomology. In Section~\ref{s:Stein}, we prove Theorem~\ref{main} under the additional assumption that there is a {\em good} resolution that 
carries a Stein structure; 
in the general case we indicate how the proof of Theorem~\ref{main} follows from some technical lemmas that are relegated to Section~\ref{s:blowups}. 
Section~\ref{s:blowups} examines lattice homology in presence of blowups and completes the proof of Theorem~\ref{main}.

\subsubsection*{Acknowledgements:} We would like to thank the referee for pointing out a mistake in the first version of the article 
(Section 5 was added to correct that mistake). The second author thanks Marco Golla and Oleg Viro for some useful discussions and comments.

\section{Resolutions and plumbing graphs} \label{s:plumbings}

Consider a resolution  $\pi: \ti{\Sigma}
\to \Sigma$ of a normal surface singularity $(\Sigma, 0)$. The inverse image in $\ti{\Sigma}$ of a small neighborhood of $0$ in $\Sigma$ is a 4-manifold 
$X$ with $\d X=Y$, where $Y$ is the link of the singularity. If the resolution is minimal (i.e. $X$ has no smooth rational curves of self-intersection $-1$), by~\cite{BO} the manifold $X_{min}$
carries a Stein structure $J$ so that $(X, J)$ is a Stein filling for the canonical contact structure $\xi_0$ on
$Y$. 

Our main tool in this paper is lattice homology, an invariant constructed from a {\em good} resolution of $(\Sigma, 0)$. A resolution is good
if the irreducible components of the exceptional divisor
$\pi^{-1}(0)$ are smooth complex curves that
intersect transversely at double points only. (In other words, $\pi^{-1}(0)$ is
required to be a normal crossing divisor.)  The dual resolution graph $\Gamma$ is the graph whose vertices
correspond   to irreducible components of the exceptional divisor
and the edges record intersections of these components.
Each vertex is decorated with an integer weight equal to the self-intersection
of the corresponding curve. The resolution yields 
a 4-manifold $X(\Gamma)$ such that $\d X(\Gamma)=Y$, where $Y$ is the link of
singularity. For normal singularities, $X(\Gamma)$ is negative-definite, and
$Y$ is a rational homology sphere if and only if $\Gamma$ is a tree and each
vertex corresponds to a 2-sphere. (See for example
\cite[§2.1-2.2]{Ne2} for details.) The manifold $X(\Gamma)$ can be
obtained by plumbing disk bundles  over 2-spheres (with Euler numbers given
by weights of vertices) as dictated by the graph $\Gamma$, so $\Gamma$ is often called a plumbing graph. If the resolution of a normal singularity $(\Sigma, 0)$ is not good, 
it can still be encoded by a graph $\Gamma$, but one needs to record the singularities of the components 
of the exceptional divisor and specify multiple 
intersections. 

It is important to note that the minimal resolution of $(\Sigma, 0)$ does not have to be good, but one 
can always obtain a good resolution from a minimal resolution by some blowups.  
(This is a standard fact, see for example~\cite[Theorem V.3.9]{Hartshorne} for a closely related result. 
See also~\cite[§1]{NemFive} for discussion and examples; bad minimal resolutions 
can appear even in very simple situations.) Accordingly, a good resolution does not have to be minimal, so $X(\Gamma)$ is not Stein in general. 
However,  $X(\Gamma)$ always carries a symplectic form
$\omega_0$ such that $(X(\Gamma), \omega_0)$ is a strong symplectic filling for
$(Y, \xi_0)$. Indeed, $\ti{\Sigma}$ is K\"ahler since it lives  in  a  blowup 
of $\CC^N$,
in  particular, $\ti{\Sigma}$   has
a  symplectic  form $\omega_0$ compatible with the complex structure. For a good resolution, the irreducible components of the exceptional
divisor $\pi^{-1}(0)$ are smooth complex curves, thus they are symplectic surfaces in $X(\Gamma) \subset \ti{\Sigma}$ with respect to $\omega_0$.
The contact manifold  $(Y, \xi_0)$ is the
convex boundary of the plumbing $X(\Gamma)$ of these symplectic surfaces.



We  work with bad minimal resolutions and their good blowups in Section~\ref{s:blowups}. We  need a minimal resolution 
to use the Stein structure, but if $\pi^{-1}(0)$ is not a normal crossing divisor, $X_{min}$ cannot be used to define the lattice homology of the link.
We will need to perform some additional blowups 
on $X_{min}$ to obtain a complex surface $X$ which is a {\em good}  resolution  with plumbing graph $\Gamma$, and examine special features of 
the lattice homology 
construction  in presence of $(-1)$ vertices of $\Gamma$.

\section{Lattice Cohomology} \label{s:latticehom}

In this section, we discuss the necessary background on lattice cohomology,
\cite{OSplumb, Ne1, Ne2}. Lattice cohomology $\HH^{\ast}(\Gamma)$
was defined by N\'emethi in \cite{Ne2} as a combinatorial theory
conjecturally parallel to Heegaard Floer homology. Starting with a plumbing
graph $\Gamma$ that defines a 4-manifold with boundary $X(\Gamma)$,  N\'emethi's
construction uses cellular cohomology of certain $CW$-complexes associated
to the lattice $L = H_2(X(\Gamma), \ZZ)$ equipped with a weight function. We do
not give the general definition of  $\HH^{\ast}(\Gamma)$   here as
we will only work with its $0$-dimensional part $\HH_0^+(\Gamma)$. (The reader
will get a glimpse of the CW-complexes in
the graded roots discussion below.)
However, we will use several equivalent definitions of the
0-dimensional cohomology, those  from \cite{Ne2} and its precursors
\cite{OSplumb, Ne1}.  We also use specific isomorphisms between these
constructions, so we review this material in some detail. (Everything we need is
contained in \cite{Ne1} but some of the statements are implicit and somewhat
difficult to extract from \cite{Ne1}.)

As before, let $Y$ be a rational homology sphere which is a link of normal
surface singularity. Let $\Gamma$ be a negative-definite connected plumbing
graph as above, defining a $4$-manifold $X = X(\Gamma)$ with boundary $\partial
X = Y = Y(\Gamma)$.


Consider the lattice $L= H_2(X, \ZZ)$;  the intersection form on $L$ can be read
off the graph $\Gamma$. 
Indeed, the vertices of $\Gamma$ give a basis for $L$; $v$ will
usually denote both a  vertex and its corresponding homology class. Then,
the self-intersection $v \cdot v$ equals the weight decoration of the
vertex $v$, and for two different vertices $v, w$ we have $v \cdot w =1$ if
$v, w$ are connected by an edge in $\Gamma$, and $0$ otherwise. 

Set $L' = H^2(X, \ZZ)$ and $H = H_1(Y, \ZZ)$. Since $X$ has no $1$-handles, 
from Poincar\'{e} duality, the universal coefficient theorem and the  homology
exact sequence of the pair $(X, Y)$ we have 
\[L' = H^2(X, \ZZ) \simeq H_2(X, Y, \ZZ) \simeq \Hom(H_2(X, \ZZ), \ZZ),\] and
our assumption that 
$Y$ is a rational homology sphere gives 
 a short exact sequence $0 \rightarrow L \rightarrow L' \rightarrow H
\rightarrow 0$. We will use the map $\PD: L \rightarrow L'$ defined by composing
the
Poincar\'{e} duality $H_2(X, \ZZ) \rightarrow H^2(X, Y, \ZZ)$ with the
cohomological inclusion $H^2(X, Y, \ZZ) \rightarrow H^2(X)$.

Let $\Char (\Gamma) \subset H^2(X, \ZZ)$ be the set of characteristic vectors,
that is,
\[\Char = \Char(\Gamma) = \{ K \in L' : \langle K, x \rangle \equiv x\cdot x \ 
(\textrm{mod\ } 2)\ \  \forall x \in L \}, \] 
where $\langle K, x \rangle$ is the evaluation of $K \in L' = H^2(X, \ZZ) \simeq
\Hom(H_2(X, \ZZ), \ZZ)$ on $x \in H_2(X, \ZZ)$ and $x\cdot x$ is the
self-intersection of $x$ by the intersection form on $L = H_2(X, \ZZ)$.

We have $\Char = K + 2L'$ for any fixed $K \in \Char$. 
The natural action $K \mapsto K+2\PD(x)$ (for any $x \in L$) of $L$ on $\Char$
has orbits of form $K + 2\PD(L)$. We will denote an orbit of this form by $[K]
\subset \Char$.

Since $X(\Gamma)$ is simply connected, $\Char(\Gamma)$ is isomorphic to  the set
of  $Spin^c$ structures on $X$, and the
identification is given by the  first Chern class of the determinant line bundle
associated with a given $Spin^c$ structure (see e.g. \cite[Proposition
2.4.16]{GS}). If $\s$ is any fixed 
$Spin^c$ structure on $X$ and  $\ts=\s|_Y$ is its restriction to $Y$,
the $Spin^c$ structures on $X$ which restrict to $\ts$ are exactly 
those whose first Chern classes form an orbit of form $[K] = K + 2\PD(L)$,
where $K = c_1(\s) \in \Char(\Gamma)$.
Thus, $Spin^c$ structures on $Y$ can be identified
with orbits of the $L$-action on $\Char(X)$, and we will sometimes use
the notation
$\ts=[K] \in Spin^c (Y)$.

For any $K \in \Char$, we will also consider the (in general, rational) number
$K^2$ defined by using the intersection
pairing on 
$H^2(X; \QQ) \simeq H_2(X, Y; \QQ) \simeq H_2 (X, \QQ)$ (the latter isomorphism holds because $Y$ is a rational homology sphere). Rational coefficients are needed 
since $H^2(X; \ZZ) \simeq H_2(X, Y; \ZZ)$ doesn't have a well-defined intersection pairing.



All the lattice cohomologies
discussed below are taken with coefficients in $\FF=\ZZ/2$ and have the structure of $\FF[U]$-modules (these modules are
graded but we omit the gradings since they will not be important to us). See
Remark~\ref{Zcoeffs} for coefficients in $\ZZ$. 

Let $\T_0^+$
denote the module $\FF[U, U^{-1}]/U \cdot \FF[U]$. We will use the notation $1=
U^{0} \in \T_0^+$ for the corresponding generator.


\subsection{Lattice cohomology via functions on {Char}.}\label{ss:latchar} 
This is a review of the construction due to
Ozsv\'{a}th and Szab\'{o}, 
\cite[\textsection 1]{OSplumb}.

Define a weight function $w$ on $\Char$ by setting $w(K) = -(K^2+|\Gamma|)/8$,
where $|\Gamma|$ stands for the number of vertices in the plumbing graph,
i.e. the number of basis elements of $H_2(X, \ZZ)$ provided by the exceptional
divisors. 

\begin{definition}\label{OS-lattice}
The 0-dimensional lattice cohomology $\HH_0^+(\Gamma) \subset \Hom (\Char (\Gamma), \T_0^+)$ 
is the set of functions satisfying the following adjunction relations 
for characteristic vectors $K \in \Char (\Gamma)$ and vertices $v$ of $\Gamma$. 
If $n$ is an integer such that $2n = \langle K, v \rangle + v \cdot v$ 
(or, equivalently, $w(K) - w(K+2\PD(v)) = n$), we require 
for every $\phi \in \HH_0^+(\Gamma) \subset \Hom (\Char (\Gamma), \T_0^+)$ that  
\begin{equation}
\label{eq:compOzsSz}
\begin{aligned}
U^n \cdot \phi(K + 2 \PD [v]) & = \phi (K) & \text{ if } n \geq 0, \\
 \phi(K + 2 \PD [v]) & = U^{-n} \cdot \phi (K) & \text{ if } n \leq 0. 
\end{aligned}
\end{equation}
\end{definition}

We introduce $U$-action on
$\HH^+_0(\Gamma)$  by setting $(U\phi)(K) = U(\phi(K))$ for every characteristic vector $K \in \Char(\Gamma)$, thus  
$\HH^+_0(\Gamma)$ becomes an $\FF[U]$-module.




As the compatibility condition \eqref{eq:compOzsSz} above involves relations 
between elements of $\Char(\Gamma)$ that differ by an element in $2\PD(L)$,  
the $\FF[U]$-module $\HH^+_0(\Gamma)$ decomposes as a direct sum according to
the $Spin^c$ structures on $Y$:
\[ \HH^+_0(\Gamma) = \bigoplus_{\ts \in Spin^c(Y)} \HH^+_0(\Gamma, \ts). \]
We will use the notation $\HH^+_0(\Gamma, [K])$ to denote the direct summand of
the above decomposition which corresponds to the $Spin^c$ structure $\ts$. Here,
$[K]$ is the $L$-orbit formed by the first Chern classes of
$Spin^s$ structures on $X$ restricting to $\ts$ on $Y$. One can think of
elements of $\HH^+_0(\Gamma, [K])$ as 
functions in $\Hom([K], \T_0^+)$
satisfying the compatibility conditions \eqref{eq:compOzsSz}.

\subsection{Lattice cohomology via functions on homology lattice.}\label{ss:lattice} We now describe
a slightly different construction by N{\'e}methi, introduced in
\cite[Proposition 4.7]{Ne1}.

Given any characteristic element $K\in \Char(\Gamma)$, define the weight
function 
$\chi_K$ on $L$ such that for any $x \in L$ we set
\begin{equation}\label{weight-fn}
 \chi_K(x) = -\frac{1}{2}(\langle K, x \rangle + x\cdot x). 
\end{equation}
\begin{definition}\label{Nem-L-lattice}
For a fixed characteristic vector $K \in \Char(\Gamma)$, the lattice cohomology 
$\HH \LL_0^+(\Gamma, K) \subset \Hom (L,\T_0^+)$ is the set 
of functions $\varphi:L\to \T^+_0$ satisfying the following 
relations for elements $x \in L$ and vertices $v$ of
$\Gamma$. If $n$ is an integer such that $2n = \langle K, v \rangle + v \cdot v
+ 2 x\cdot v$, or, equivalently, if 
$\chi_K(x) - \chi_K(x+v) = n$
we require that  
\begin{equation}
\label{eq:compNem}
\begin{aligned}
U^n \cdot \varphi(x + v) & = \varphi (x) & \text{ if } n \geq 0, \\
 \varphi(x + v) & = U^{-n} \cdot \varphi (x) & \text{ if } n \leq 0. 
\end{aligned}
\end{equation}
\end{definition}

This is also naturally an $\FF[U]$ module by setting $(U\varphi)(x) =
U(\varphi(x))$.



\begin{lemma} \cite[Proposition 4.7]{Ne1}
\label{lem:charvsl}
\[ \HH \LL_0^+(\Gamma,K) \cong \HH_0^+(\Gamma, [K])
. \]
\end{lemma}
\begin{proof}
The isomorphism is constructed as follows.
Let $\iota_K: L \rightarrow [K] = K + 2\PD(L) \subset \Char$ be the mapping
defined by $\iota_K(x) = K + 2\PD(x)$.
Let $\iota_K^*: \HH^+_0(\Gamma, [K]) \rightarrow \HH\LL^+_0(\Gamma,K)$ be the
induced
dual map, that is, 
$\Hom(L, \T_0^+) \ni \varphi = \iota_K^*(\phi)$ for $\phi \in \Hom([K], \T_0^+)$
if $\varphi = \phi \circ \iota_K$.
This map is well-defined as the two compatibility conditions
\eqref{eq:compOzsSz} and \eqref{eq:compNem} 
correspond to each other: setting $K' = K + 2\PD(x)$, we see that for a 
basis element $v$ of $L$ corresponding to a vertex of the plumbing graph,
\[ \langle K', v \rangle + v \cdot v = \langle K, v \rangle + v \cdot v +
2x\cdot v. \]
In the language of the weight functions, this is exactly the fact 
$w(\iota_K(x+v)) - w(\iota_K(x)) = \chi_K(x+v) - \chi_K(x)$.
\end{proof}


\subsection{Lattice cohomology via graded roots}\label{ss:grr}
Here we review N\'emethi's  main construction from \cite[§ 4]{Ne1}.

Fix $K\in \Char (\Gamma)$ and again consider the weight function $\chi_K: L \to
\ZZ$ defined
by equation \eqref{weight-fn} above.
We consider sublevel sets of the function $\chi_K$ in the lattice~$L$.
For each $n \in \ZZ$, let $\bar{L}_{K, \leq n}$ be  a finite
1-dimensional cell complex whose $0$-skeleton is the set 
$$
L_{K, \leq n} = \{x
\in L: \chi_K(x) \leq n \},
$$ and the $1$-cells are constructed as follows. If
$x \in
L$ and $v$ is the basis element of  $L=H_2(X, \ZZ)$ corresponding to a
vertex of $\Gamma$, then we connect $x$ and $x+v$  by a unique
1-cell in $\bar{L}_{K, \leq n}$  whenever  $x$ and $x+v$ are both in $L_{K, \leq
n}$. Clearly, such cell complexes can be built as subsets of $L \otimes \RR$,
taking the 1-cells to be straight line segments connecting their endpoints. Then
we have  $\bar{L}_{K, \leq n} \subset  \bar{L}_{K, \leq m}$ for $n<m$.

Consider the set $\pi_0(\bar{L}_{K, \leq n})$ of the connected
components of  $\bar{L}_{K, \leq n}$, and let $C_w$ denote the 
component corresponding to $w \in \pi_0(\bar{L}_{K, \leq n})$. If $m>n$, each
$C_w$ is contained in a component $C_{w'}$ for some  $w' \in  \pi_0(\bar{L}_{K,
\leq m})$, and $C_{w'}$ may contain several distinct components of 
$\bar{L}_{K,
\leq n}$. These inclusion relations are codified by the {\em graded root} 
$(R_K, \chi_K)$, which is a graph with an integer-valued  grading function. 
The grading on the graph is closely related to the $U$-action on cohomology.

The vertices $\V(R_K)$  of $(R_K, \chi_K)$ are given by the set
$\cup _{n \in \ZZ} \pi_0(\bar{L}_{K, \leq n})$. The grading, $\V(R_K) \to \ZZ$,
still denoted by $\chi_K$, is defined by 
$\chi_K|_{\pi_0(\bar{L}_{K, \leq n})}=n$. 
Finally, all edges are obtained by connecting vertices of the form
$w_n \in \pi_0(\bar{L}_{K, \leq n}) $ and $w_{n+1} \in \pi_0(\bar{L}_{K, \leq
n+1})$ such that $C_{w_n} \subset C_{w_{n+1}}$, where the inclusion is
understood in the sense described above.

\begin{remark}\label{rem:rootgrading}
As we mentioned, the elements of $\Char(\Gamma)$ fall into equivalence classes of
form $[K]$ corresponding to $Spin^c$ structures on $Y$.
 It turns out that the graded roots corresponding to two
 characteristic elements $K, K'$ belonging to the same orbit (that is, $K - K' \in 2\PD(L)$) 
 are the same up to a grading shift,  so
one can associate a well-defined graded root $(R_{\ts}, \chi_{\ts})$
to a $Spin_c$ structure $\ts \in Spin^c(Y)$ if one fixes the grading so that
$\min \chi_{\ts}|_{R_{\ts}}=0$, see \cite[Section 4]{Ne1} for
details. 
As we do not work with absolute gradings on cohomology modules, we will not make the grading shift 
and will simply use the grading given by $\chi_K$.
\end{remark}

\begin{definition} Fix a characteristic element $K \in \Char$, let $\chi_K$  be
a weight function as in~(\ref{weight-fn}), and   consider the graded root  
 $(R_K, \chi_K)$ as above with the vertex set $\V= \V(R_K)$. 
The associated
$\FF[U]$ module $\HH(R, \chi)$ is  defined as the set of functions $\psi:\V \to
\T_0^+$ satisfying the condition 
\begin{equation}\label{H-condition}
U \cdot \psi(v) = \psi(w) \text{ if } v,w \text{ are
connected by an edge of } R \text{ and } \chi(v)< \chi(w).
\end{equation}
\end{definition}

Note that by the construction of the graded root, for $v, w$ as above
we have in fact $\chi(v) + 1 = \chi(w)$. As before, there is obvious $U$-action
on  $\HH(R, \chi)$, so that 
 $(U\psi)(v) = U(\psi(v))$. See \cite[Definition 3.5]{Ne1} and discussion
therein for details. 





\begin{lemma}\cite[Proposition 4.7]{Ne1}
\label{lem:latticevsgradedroot}
\[ \HH\LL_0^+(\Gamma, K) \cong \HH(R_K, \chi_K). \]
\end{lemma}
\begin{proof}

The isomorphism of \cite[Proposition 4.7]{Ne1} is constructed as follows. For
an element $x \in L$ with $\chi_K(x)=n$, 
the map $\theta: L \rightarrow \mathcal{V}(R_K)$ associates to
$x$   the component of $\pi_0(\bar{L}_{K,\leq n})$
containing
$x$.

This induces a map $\theta^*: \Hom(\mathcal{V}(R_K), \T_0^+) \rightarrow \Hom(L, \T_0^+)$ given by $\theta^*(\psi) = \varphi$ if $\varphi = \psi \circ \theta$. 
One can check that this is indeed a well-defined mapping from $\HH(R_K, \chi_K)$
to $\HH\LL_0^+(\Gamma, K)$, as the compatibility conditions \eqref{eq:compNem} and
\eqref{H-condition} are matching. By some more work, it is also easy to see that
it is an isomorphism (for the details, see the proof of \cite[Proposition
4.7]{Ne1}). 
\end{proof}

We will also need a special property of the graded root corresponding to the
canonical class. Some care with notation and terminology is needed here: in \cite{Ne1},  N\'{e}methi defines the
canonical class $K_{can}$ as the first Chern class of the canonical line bundle, also uniquely determined  by the relations
\[ \langle K_{can}, v \rangle = - v \cdot v - 2 \]
for every basis element of $L$ corresponding to a vertex of $\Gamma$. We work instead with the anticanonical class  $K_0= c_1(TX, J)$.
If 
$\s_{can}$ is the  $Spin^c$ structure on $X$ induced by $J$, we have 
$K_0=c_1(\s_{can})$. Note that $\s_{can}$ is usually referred to as the canonical $Spin^c$ structure; its restriction to the 3-manifold $Y$ is the  $Spin^c$ structure
on $Y$ induced by the canonical contact structure $\xi_0$. We have 
\[ \langle  c_1(TX, J), v \rangle =  v \cdot v + 2 \] by adjunction,
and the relation between our  class $K_0$ and N\'{e}methi's
canonical class $K_{can}$ is $K_0 = - K_{can}$.  We will use \cite[Theorem 6.1(c, d)]{Ne1}, which  N\'{e}methi proves 
for $K_{can}$. However, $\chi_{K_0}(x) = \chi_{K_{can}}(-x)$ for any $x \in L$, and this symmetry implies that any
statements about the connected components of level sets with respect to these
two weight functions will be the same (cf. \cite[section 5.1]{Ne1}).
This enables us to state the lemma below for the class $K_0$.

\begin{lemma} \cite[Theorem 6.1(c, d)]{Ne1} \label{can-vertex}   Let $K_0 = c_1(TX, J) \in \Char(\Gamma)$. 

(1) Consider the sublevel set $\bar{L}_{K_0, \leq 0}$, and  let $C_0$ be its
connected component containing $0 \in L = H_2(X, \ZZ)$.
Then $C_0$ contains no points $x$ with $\chi_{K_0}(x)<0$, i.e. $\chi_{K_0}|_{C_0}$
is identically zero.  

(2) The sublevel set $\bar{L}_{K_0, \leq n}$ is connected for $n\geq 1$. 

(3) The graded root $(R_{K_0}, \chi_{K_0})$ has a distinguished vertex $w_0$ of valency one, which is the end vertex of an infinite (sub)chain consisting of vertices $w_0, w_1, w_2, \dots$ such that
$\chi_{K_0}(w_i) = i$ and there is an edge between $w_i$ and $w_{i+1}$ for every $i \in \mathbb{Z}^+_0$. Moreover, for every $i > 0$, the only vertex $v$ of the graded root with $\chi_{K_0}(v) = i$ is $v = w_i$. 
\end{lemma}

The third part of the above lemma directly follows (using the construction of the graded root) from the first two parts which are explicitly stated in \cite[Theorem 6.1(c, d)]{Ne1}. The distinguished vertex $w_0$ is the connected component $C_0$ containing $0 \in L$ in $\pi_0(\bar{L}_{K_0,\leq 0})$, and the vertex $w_i$ for $i > 0$ is the single connected component of the connected sublevel set $\bar{L}_{K_0, \leq i}$.

We will call the infinite (sub)chain $w_0, w_1, w_2, \dots$ the \emph{main trunk} of the canonical graded root $(R_{K_0}, \chi_{K_0})$.
Note that the canonical graded root in general can have many complicated branches outside the main trunk (if the singularity is not rational, see the proof of Lemma~\ref{psi0} later), but those other branches, if present, connect to the main trunk at the level-one vertex $w_1$, see Figure~\ref{root-pic}.

\begin{center}
\begin{figure}[tb]
\includegraphics[width=0.4\textwidth]{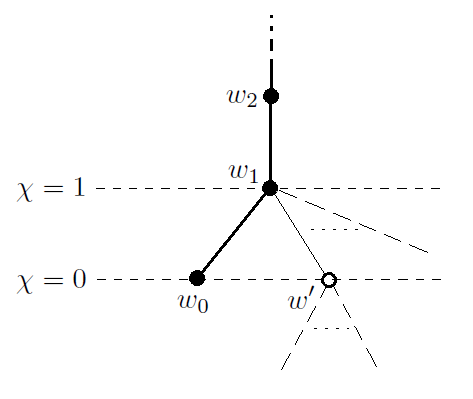}
\caption{A sketch of a graded root with its main trunk. At least one vertex $w'$
not on the main trunk is present on the 0-level, if the singularity is not
rational.} \label{root-pic}
\end{figure}
\end{center}



\section{The contact invariants: the Stein filling case} \label{s:Stein}

In this section, we prove Theorem~\ref{main}  in the case where the minimal resolution $X_{min}$ has a good graph $\Gamma$.
We build on ideas from \cite{Ka, KO}: in these papers, 
Karakurt and Karakurt-\"Ozt\"urk consider the case  where the Heegaard Floer homology is isomorphic to
$\HH_0^+(\Gamma)$ (namely, they work with AR-graphs, for the definition see
\cite[§8]{Ne1}),  and study  the image of the
contact invariant $c^+(\xi) \in HF^+(-Y)$  in the lattice homology under this
isomorphism. Even when the isomorphism no longer holds, we are able to apply a similar strategy.

First, we state a  lemma for an arbitary negative-definite plumbing graph $\Gamma$ (the minimality assumption is not needed yet).   
Let $W(\Gamma)$ be the cobordism from $S^3$
to $Y=Y(\Gamma)$ given by the plumbing graph; $W(\Gamma)$ is obtained by cutting a
small ball out of $X(\Gamma)$; $Spin^c$ structures on  $W(\Gamma)$ are
naturally identified with those on $X(\Gamma)$, and in turn with
$\Char(\Gamma)$. We can think of $W(\Gamma)$ as cobordism from $-Y$ to $S^3$. Let
$F^+_{W(\Gamma), \s}:
HF^+(-Y) \to HF^+(S^3)$ be the map on Heegaard Floer homology induced by the
$Spin^c$ cobordism  $(W(\Gamma), \s)$
(see \cite{OSfour}). Following \cite{OSplumb}, define the map 
$$
T^+: HF^+(-Y) \to \HH_0^+(\Gamma) 
$$
as follows: for $x \in   HF^+(-Y)$, let $T^+(x): \Char (\Gamma) \to \T_0^+$ be
given by 
$$
T^+(x)(K)= F^+_{W(\Gamma), \s}(x) \in HF^+(S^3)= \T_0^+,
$$
where $K$ is the element of $\Char$ associated with the $Spin^c$ structure $\s$.

\begin{lemma} \label{OS-map} \cite[Proposition 2.4]{OSplumb} The map $T^+$
induces an   $\FF[U]$-equivariant 
map from $HF^+(-Y(\Gamma), \ts)$ to $\Hom (\Char_{\ts}(\Gamma), \T_0^+)$, whose
image
lies in  $\HH_0^+(\Gamma, \ts)$. 
\end{lemma}

Again, it's important to note that the above lemma only uses
basic properties of Heegaard Floer cobordism maps and requires no additional
assumptions on the negative-definite graph $\Gamma$. (This map is an isomorphism for
AR-graphs, see \cite[Theorem 8.3]{Ne1}.)

We would like to find an explicit lattice-homological description of the element $c=T^+(c^+(\xi_0))\in \HH_0^+(\Gamma)$.
In the next lemma, we will do so under
the additional assumption that  the graph $\Gamma$ contains no $(-1)$ vertices, i.e. 
$X(\Gamma)$ is minimal and thus gives a Stein filling of $Y(\Gamma)$ (see Section~\ref{s:plumbings}).
In the next section, we will expand the argument to the general case.

\begin{lemma}
\label{lem:canhom-min} Assume that the graph $\Gamma$ contains no $(-1)$ vertices, i.e. $X(\Gamma)$ is Stein. 
Consider the element $c=T^+(c^+(\xi_0))\in \HH_0^+(\Gamma)$. 
Let $K_0 = c_1(TX, J)$. 
(1) The element $c$ is given by  a function 
$\phi_0 \in \HH_0^+(\Gamma)$ such that
$\phi_0(K_0)=1$ in the degree $0$ part of $\T_0^+$ and $\phi_0(K)=0$ for any
other characteristic class $K \neq K_0$. 
In particular, $c \in \HH^+_0(\Gamma, [K_0])$.

(2) Under the isomorphism $i_{K_0}$ of Lemma~\ref{lem:charvsl}, the function
$\phi_0$ corresponds to the  function $\varphi_0 \in \HH\LL_0^+(\Gamma, K_0)$
such that $\varphi_0(0) = 1 \in \T^+_0$ in degree $0$ and $\varphi_0(x) = 0$  
for any $x\neq 0$, $x\in L$.

(3) Under the isomorphism $\theta^*$ of Lemma~\ref{lem:latticevsgradedroot}, 
the function $\varphi_0$ corresponds to the function $\psi_0: \V(R_{K_0}) \to
\T_0^+$
such that $\psi_0(w_0)$ is the generator of $\T_0^+$ in degree $0$, and
$\psi_0=0$
for all other vertices of $R_{\ts_{can}}$. The vertex $w_0$ corresponds to the connected component of $0$ in the homology lattice.
\end{lemma}

\begin{proof}
The argument is essentially the same as Karakurt's observation in \cite{Ka},
based on the main theorem of \cite{Pla}.
Indeed, the homomorphism $c=T^+(c^+(\xi_0))$ is defined by 
$$
c(K)= F^+_{W(\Gamma), K}(c^+(\xi_0)) \in HF^+(S^3)= \T_0^+.
$$
By \cite{OScont},  $c^+(\xi_0) \in HF^+(Y, \s_{can}|_{Y})$ , 
so it follows immediately that $c \in
\HH^+_0(\Gamma, [K_0])$ since
the map $T^+$ respects $Spin^c$ structures.  

By our assumption,  $X(\Gamma)$ carries a Stein structure $J$ so
that $(X, J)$ is a Stein filling for the canonical contact structure $\xi_0$ on
$Y$. In this case   
\cite[Theorem 4]{Pla} asserts that 
$F^+_{W(\Gamma), \s_{can}}(c^+(\xi_0))$ is the generator
of $\T_0^+$ in degree $0$, and  $F^+_{W(\Gamma), \s}(c^+(\xi_0))=0$ for
any other $Spin^c$-structure $\s$ on $W(\Gamma)$. Since $c_1(\s_{can})=  c_1(TX, J)=K_0$ by
our definition of the class $K_0$, this means that $c=\phi_0$.

Parts (2) and (3) of the lemma are immediate from the definitions of
isomorphisms $\HH_0^+(\Gamma, [K_0])
\simeq
\HH\LL_0^+(\Gamma, K_0)
\simeq \HH(R_{K_0}, \chi_{K_0})$ of Lemmas~\ref{lem:charvsl} and 
\ref{lem:latticevsgradedroot}.
\end{proof}

\begin{remark} \label{rem:C=0} Under the hypothesis that $\Gamma$ contains no $(-1)$ vertices, the connected component $C_0$ of $0$ in the  level set $\overline{L}_{K_0, \leq 0}$
consists of a single point, ie $C_{0}=\{ 0 \}$. Indeed, due to the isomorphism of Lemma \ref{lem:latticevsgradedroot}, the function $\varphi_0$ 
must vanish on the entire component $C_0$. It is also easy to check that $C_{0}=\{ 0 \}$ directly as follows.  With the 
notation of Section~\ref{ss:lattice},  if $v$ is a basis element, 
we use the identity $\langle K_0, v \rangle =  v \cdot v  + 2$ to see that
\[ \chi_{K_0}(x) - \chi_{K_0}(x \pm v) = \pm  x \cdot  v  + 
\begin{cases}
v \cdot  v  + 1 \textrm{\   or} \\
 -1.
\end{cases} \]
If $\Gamma$ has no $(-1)$ vertices, the inequality $ v \cdot v  \leq -2$ holds for any basis element $v$. 
Then  for $x = 0$, we get $\chi_{K_0}(0) = 0$ but $\chi_{K_0}(\pm v) > 0$ for any basis element $v$, so $0 \in L$ is a single point in its 
connected component of the level set  $\overline{L}_{K_0, \leq 0}$. 
\end{remark}

In the next section, we examine the general case where $X(\Gamma)$ may not be Stein.  We will see that  Parts~(1) and (2) of Lemma~\ref{lem:canhom-min} no longer hold. 
However, it turns out that the zero set of the function  $\varphi_0$ still matches the connected component $C_0$, so that Part~(3) of Lemma~\ref{lem:canhom-min}
holds in general. See Section~\ref{s:blowups}.

We now return to graded roots to establish a useful property of the function $\psi_0$.

\begin{lemma} \label{psi0} Let $w_0$ be the distinguished vertex of the canonical graded root $(R_{K_0}, \chi_{K_0})$ in the sense of Lemma~\ref{can-vertex}.
Consider  $\psi_0: \V(R_{K_0}) \to \T_0^+$
such that $\psi_0(w_0)=1 \in \T_0^+$  and
$\psi_0(v)=0$
for all other vertices $v$ of $R_{K_0}$. Then $\psi_0 \in \HH(R_{K_0},
\chi_{K_0})$, and $\psi_0 \in \Ker U$. Moreover,  $\psi_0 \in \Im U$ if
and only if the singularity is rational. 
\end{lemma}
\begin{proof} We need to check that $\psi_0$ satisfies the compatibility
conditions \eqref{H-condition} which is immediate because the 
generator 
$1 \in \T_0^+$ is annihilated by $U$ and by Lemma~\ref{can-vertex} part (1),
there is no vertex $v$ of the graded root connected to $w_0$ such that $\chi(v)
< \chi(w_0)$ ($w_0$ is  valency-one vertex of the graded root). 
Similarly, $\psi_0 \in \Ker U$ follows from the relations \eqref{H-condition}.
(Alternatively, one can use the fact that by Lemma~\ref{lem:canhom-min} and Lemma~\ref{lem:component}, $\psi_0$ is the image of $c^+(\xi_0)$
under the map from Heegaard Floer to lattice homology, and $c^+(\xi_0)\in \Ker U$ in Heegaard
Floer homology by \cite{OScont}.)

By \cite[Theorem 6.3]{Ne1}, the singularity is rational if and only if $\HH(R_{K_0}, \chi_{K_0}) = \T_0^+$, 
and this happens exactly when the graded root is a single infinite chain with the end vertex $w_0$, that is, the graded root consists of nothing else but the main trunk.

Therefore, if the singularity is rational, it is easy to see that $\psi_0 \in
\Im U$. 
If the singularity is not rational, the graded root $R_{K_0}$ has non-trivial
branches, i.e., at least one vertex $v \neq w_i\  (i \in
\mathbb{Z}^+_0)$  outside its main trunk. 
Recall that $w_1$ is the (unique) vertex
connected to $w_0$ by an edge in $R_{K_0}$ and $\chi_{K_0}(w_1)=1$.
By Lemma~\ref{can-vertex}, all the vertices not on the main trunk must have non-positive $\chi_{K_0}$-value, so there exists
a vertex $w'\neq w_0$ such that $\chi(w')=0$ and $w'$ is connected to $w_1$, see  
Figure~\ref{root-pic}.
 Now, suppose that $\psi_0 = U\psi$ for some 
$\psi \in \HH(R_{K_0})$. Then $\psi(w_1)= U \psi(w_0)= \psi_0(w_0) \neq
0$ and $\psi_0(w')= U \psi(w')= \psi(w_1) \neq 0$, which is a contradiction
because we defined $\psi_0(v)=0$ for all vertices $v \neq w_0$. 
\end{proof}

\begin{remark}
The above argument is similar to \cite[§5.8]{KO}, but  
Karakurt-\"Ozt\"urk in  \cite{KO} use Laufer
sequences, an approach that only works in the special case of AR-singularities. 
Instead, we rely on the general graded root defined in \cite{Ne1} for any
negative definite rational homology sphere plumbed manifold.
\end{remark}

\begin{proof}[Proof of Theorem \ref{main}] The result follows immediately from
Lemmas~\ref{psi0},~\ref{lem:canhom-min},~\ref{OS-map} in the Stein case. Similarly, in the general case the proof follows 
from   Lemmas~\ref{psi0},~\ref{OS-map}, and Lemma~\ref{lem:component} established in the next section.
Indeed, for the canonical $Spin^c$ structure~$\ts$ there is an
$\FF[U]$-equivariant map 
  $HF^+(-Y(\Gamma), \ts) \to \HH_0^+(\Gamma, [K_0]) \simeq \HH\LL_0^+ (\Gamma,
K_0)
  \simeq \HH (R_{K_0}, \chi_{K_0})
  $ 
mapping
$c^+(\xi_{0})$ to $\psi_0$.
Therefore, if $\psi_0 \notin \Im U$, then
$c^+(\xi_{0})\notin \Im U$. 
\end{proof}


\section{The contact invariants in presence of blowups} \label{s:blowups}


In this section, we address the case where the minimal resolution $X_{min} \to \Sigma$ of the singularity $(\Sigma, 0)$ does not have normal crossings. 
As discussed in Section~\ref{s:plumbings}, we can perform some additional blowups to obtain a normal crossing resolution $X=X(\Gamma)$, so that 
the resolution graph $\Gamma$ is good.  
The blowups are performed in several steps; at every step we blow up one or several points simultaneously. This gives a sequence of complex surfaces
\begin{equation} \label{blowupdown}
X=X_1 \to X_2 \to  X_3 \to \dots X_{min} \to \Sigma,
\end{equation}
where the maps are the corresponding blowdowns. It will be convenient, even if not strictly necessary, to assume that blowdowns are performed simultaneously whenever the surface contains 
several  rational curves with self-intersection $-1$.  (Note that in a surface with negative-definite intersection form, two distinct complex
curves with self-intersection $-1$ must be disjoint.)  The graph $\Gamma$ encodes the exceptional divisor of the composite map 
$X \to \Sigma$.  If $\ti Z \to Z$ is a blowup at a single point, we have $H_2(\ti Z) = H_2(Z) \oplus H_2(\overline{\CP^2})$, where the second summand corresponds to the exceptional 
divisor, and the map $H_2(Z) \hookrightarrow H_2(\ti Z)$ is induced by the inclusion of the punctured copy of $Z$ into $\ti Z$. 
Thus, we have homology inclusions $H_2(X_r) \hookrightarrow H_2(X)$ for each $r$. 

We introduce notation for the components of the exceptional divisors at different stages of the blowup~(\ref{blowupdown}), as follows. Let $D^1_1, D^1_2, \dots, D^1_{k(1)}$ 
be the collection of disjoint rational curves of self-intersection $-1$ in $X_1$ that are blown down to obtain $X_2$,  
and  let $E^1_1, E^1_2, \dots, E^1_{k(1)}$ denote the corresponding vertices of $\Gamma$. 
At the second step, there is a collection of disjoint exceptional curves $D^2_1, D^2_2, \dots, D^2_{k(2)}$ in $X_2$;  these are 
simultaneously blown down to obtain $X_3$. Under the blowup  $X_1 \to X_2$, the strict transforms 
of $D^2_1, D^2_2, \dots, D^2_{k(2)}$ become components of the exceptional divisor in $X_1$  encoded by $\Gamma$, so they correspond to certain 
vertices $E^2_1, E^2_2, \dots, E^2_{k(2)}\in \Gamma$. (Under our assumption,  $D^2_j$ could not be blown down in $X_1$, so 
$E^2_j \cdot E^2_j \leq -2$, and the self-intersection increases in $X_2$ as a result of the blowdown.) 
Inductively, we blow down a collection  $D^r_1, D^r_2, \dots, D^r_{k(r)}$ of rational curves with self-intersection $-1$ in $X_r$ to obtain the surface $X_{r+1}$, 
until we arrive to the surface $X_{R+1}=X_{min}$ after $R$ steps. For each surface, we have a map $X\to X_r$ given by~(\ref{blowupdown}); for the 
divisors $D^r_1, D^r_2, \dots, D^r_{k(r)} \subset X_r$, their strict transforms  in $X$ correspond 
to some components of the exceptional divisor of the map  $X \to \Sigma$. This exceptional divisor is encoded by $\Gamma$; let $E^r_1, E^r_2, \dots, E^r_{k(r)}$ denote 
the vertices of $\Gamma$ that correspond to (the strict transforms of) $D^r_1, D^r_2, \dots, D^r_{k(r)}$.

Recall that the collection of divisors in $X$ corresponding to all vertices of the graph $\Gamma$ gives a basis of $H_2(X)$. 
Some, but not all, of these vertices appear in the sets $\{E^r_1, E^r_2, \dots, E^r_{k(r)}\}$
above. The vertices that do not appear in these sets correspond to classes that come from  $H_2(X_{min})$. 

\begin{remark} By construction, the smooth complex curve  $D^r_j$ lies in $X_r$. Its preimage under  the map 
$X \to \dots \to X_r$ is the total transform $\ti D^r_j \subset X$; generally, this divisor is reducible. 
The image of the 
homology class of $D^r_j$ under the inclusion $H_2(X_r) \hookrightarrow H_2(X)$ is the homology class of the total transform; we use the same notation $D^r_j$ for this
class in $H_2(X)$. Note that when the total transform consists of several components and is not smooth, 
this class  cannot be realized by a smooth complex curve in $X$: if $D$ were a smooth complex representative, 
then $D$ would non-negatively intersect all components of the total transform  $\ti D^r_i$, contradicting $D \cdot \ti D^r_i= D \cdot D = -1$.
By contrast, each class $D^r_j$ can be realized by a smooth {\em symplectic} sphere in $X$ with self-intersection $-1$, so that all these symplectic spheres 
(for all $r, j$)  are 
pairwise disjoint. (Recall that the complex surface  $X$ has a compatible symplectic structure $\omega_0$, see section~\ref{s:plumbings}.) 
Indeed, we start with the first blowup in $X_{min}$: for the map $X_R \to X_{min}=X_{R+1}$, the exceptional  
divisors  $D^R_1, D^R_2, \dots, D^R_{k(R)}$ in $X_R$ are disjoint smooth complex (and thus symplectic) curves. Next, we blow up $X_R$ at one or more points to obtain 
$X_{R-1}$.  If any of the blown-up points lie in the curves  $D^R_1, D^R_2, \dots, D^R_{k(R)}$ in $X_R$, we can push each such divisor $D^R_j$ off these points by a
$C^{\infty}$-small isotopy in $X_R$. 
Since being a symplectic surface is an $C^{\infty}$-open condition, the perturbed curves will be symplectic. 
These symplectic spheres are now unaffected 
by the blowups  and thus remain smooth in $X_{R-1}$; they are obviously disjoint from the new exceptional divisors $D^{R-1}_j$ in $X_{R-1}$ 
(the curves $D^{R-1}_j$ are the 
preimages of the blown-up points under $X_{R-1} \to X_R$). We continue this process: before blowing up $X_{R-1}$ to get $X_{R-2}$,
we perturb any of the affected spheres  ($D^R_j$ or $D^{R-1}_j$) off the blown-up points, and so on, until we arrive at $X=X_1$ and a collection of disjoint symplectic 
spheres representing homology classes $D^r_j \in H_2(X)$ for all $r,j$.
\end{remark}

Let 
\begin{equation}\label{Dclasses}
\mathcal{D} = \{ D^1_1, D^1_2, \dots, D^1_{k(1)},   D^2_1, D^2_2, \dots, D^2_{k(2)}, \dots,  D^{R}_1, D^{R}_2, \dots, D^R_{k(R)} \} \subset H_2(X)
\end{equation}
be the set of homology classes $D^m_j$ in $X$ costructed above. Let $\mathcal{S} \subset H_2(X)$  be the set of sums 
$D_{i_1}+ D_{i_2}+ \dots+ D_{i_r}$
of distinct elements from $\mathcal{D}$,
where $\{D_{i_1}, D_{i_2}, \dots D_{i_r}\}$ ranges over subsets of $\mathcal{D}$. Note that $\mathcal{S}$ contains 0 (which corresponds to the empty subset). 
In other words, 
\begin{equation} \label{Sclasses}
\mathcal{S} = \{ \sum_{n,i} \varepsilon^n_i D^n_i: D^n_i\in\mathcal{D}, \varepsilon^n_i=0,1 \} \subset H_2(X). 
\end{equation}

We will need to express the homology classes in $\mathcal D$ in terms of the basis elements $E^n_i \in H_2(X)$ corresponding to the vertices of $\Gamma$. 
It is convenient to use the notion of proximity (we tweak the usual definition of proximity of points to talk about proximity of divisors). 
Consider  vertices $E^a_i, E^b_j\in  \Gamma$, with $b\geq a$.  For the curve $E^a_i$ in $X=X_1$,  we can consider its images under projection to the blowdown surfaces 
$X_2, X_3, \dots, X_a$. In $X_a$, the image is the curve $D^a_i$, it 
has self-intersection $(-1)$ and gets blown down at the subsequent step. Similarly, $E^b_j$ projects non-trivially to $X_1, X_2, \dots, X_b$ 
and becomes a point after the blowdown to $X_{b+1}$.  If $D^a_i$ intersects the image of 
$E^b_j$ in $X_a$, we say that $E^a_i$ is {\em proximate} to $E^b_j$ and use notation $E^a_i \rightsquigarrow E^b_j$.
Equivalently, 
$E^a_i \rightsquigarrow E^b_j$ if, once  $E^a_i$ becomes a point  in $X_{a+1}$ after the blowdown $X_a \to X_{a+1}$, this point lies in 
the image of $E^b_j$ in $X_{a+1}$. (Note that $E^a_i \rightsquigarrow E^b_j$ implies 
in fact that $b>a$:  if 
$b=a$, then the projections of $E^i_a$ and  $E^b_j$ to $X_{a}=X_b$ are the smooth complex curves $D^a_i$ and $D^b_j$ with self-intersection $(-1)$, 
which must be disjoint in a negative-definite manifold.)

\begin{lemma} The homology classes $D^r_j \in H_2(X)$ can be recursively expressed via the basis classes $E^m_i$ as follows:
\begin{equation}\label{eq:prox}
D^m_j = E^m_j + \sum_{(i,n): E^n_i \rightsquigarrow E^m_j}  D^n_i.
\end{equation}
\end{lemma}
\begin{proof} The lemma follows from the familiar relation between the homology classes of the strict transform and the total transform. 
Indeed, blowing up $X_{m}$ to obtain 
$X_{m-1}$, we have
$$
[\text{strict transform of } D^m_j \text{ in  } X_{m-1} ]= D^m_j - \sum_{i: E^{m-1}_i \rightsquigarrow E^m_j}  D^{m-1}_i,
$$
because by definition, the proximate classes $E^{m-1}_i$ correspond precisely to exceptional divisors $D^{m-1}_i \subset X_{m-1}$ that project to points in  $D^m_j$
under blowdown $X_{m-1} \to X_{m}$. Note that $D^m_j$ is a smooth complex curve in $X_{m}$, so all these points are smooth, and all the divisors enter with multiplicity 
1 in the above formula for the strict transform.  

Further, observe that the strict transform of  $D^m_j$  in   $X_{m-1}$ (taken with respect to the map  $X_{m-1} \to X_{m}$) 
is the same as the projection of $E^m_j$ under the blowdown $X=X_1 \to \dots \to X_{m-1}$. It follows 
that the vertices $E^{m-2}_i$ proximate to $E^m_j$ correspond exactly to those divisors $D^{m-2}_i$ among  $D^{m-2}_1, \dots, D^{m-2}_{k-2}$ in $X_{m-2}$ that project,
under the blowdown  $X_{m-2} \to X_{m-1}$, to points of this strict transform of $D^m_j$. All these points are smooth. Comparing the strict transform in $X_{m-2}$ to 
the total 
transform, we get 
$$
[\text{strict transform of } D^m_j \text{ in  } X_{m-2} ]= D^m_j - \sum_{i: E^{m-1}_i \rightsquigarrow E^m_j}  D^{m-1}_i -
\sum_{i: E^{m-2}_i \rightsquigarrow E^m_j}  D^{m-2}_i. 
$$
We continue inductively, obtaining similar expressions for strict transforms  of $D^m_j$ in $X_{m-3}$, etc, until we arrive at the formula for $E^m_j$, which is a strict
transform of $D^m_j$ in $X=X_1$:
$$
E^m_j = D^m_j - \sum_{(i,n): E^n_i \rightsquigarrow E^m_j}  D^n_i.
$$
\end{proof}

Explicit calculations (such as Example~\ref{ex:div} below) can be done via handleslides on the plumbing graph, Kirby calculus-style. 
The graph $\Gamma$ gives a Kirby diagram for $X_1$ (each vertex corresponds to an unknotted component of the link). At each step of the blowdown, we look for 
$(-1)$-framed unknots in the diagram, handleslide them away, and blow down; the classes in $H_2$ corresponding to these $(-1)$ unknots in the diagram for 
$X_r$ are $D^r_1, \dots, D^r_{k(r)}$.  
Proximity relation $E^a_i \rightsquigarrow E^b_j$ means that in the diagram for $X_a$,  the $(-1)$-framed unknot  representing $D^a_i$  is linked with the component 
that corresponds to $E^b_j$ (this component survives in $X_a$ because $b>a$). Accordingly, when we handleslide $D^a_i$ away from a link component, the homology class 
of the handle corresponding to the latter changes by adding $D^a_i$. It is not hard to see that in this context,  Formula~(\ref{eq:prox}) is a consequence of handle 
addition operations and their effect on homology.  Note that if  we handleslide away the components corresponding to $D^m_j$ but 
do not blow them down, the resulting diagram represents $X$ at every step, and we explicitly see smooth representatives of the homology classes from 
$\mathcal D$ in $X$, given by all the $(-1)$ framed unknots that we slide away.

\begin{center}
\begin{figure}[htb]
\includegraphics[width=0.3\textwidth]{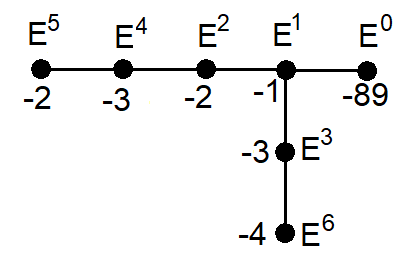}
\caption{Plumbing graph of Example \ref{ex:div}}
\label{figex}
\end{figure}
\end{center}

\begin{example}\label{ex:div}
Consider the negative definite plumbing graph of Figure \ref{figex}.
(This is the plumbing representation of the $(-1)$-surgery along the $(8,11)$ algebraic torus knot: the part of the graph obtained by deleting $E^0$ is the plumbing of $S^3$ arising as the embedded good resolution of the plane curve singularity $x^8 = y ^{11}$.)

In this example, at every blowdown step, we encounter one $(-1)$-divisor only, therefore, we omit the lower indices.
The divisors $E^1, E^2, \dots, E^6$ can be blown down in this order consecutively.
Just before $E^1$ is blown down, it intersects  $E^3$ and $E^2$, so $E^1$ is proximate to $E^3$ and also to $E^2$. 
Next, $E^2$ is blown down, and just before this happens in $X_1$, it intersects the strict transforms of $E^4$ and $E^3$, 
so $E^2$ is proximate to $E^4$ and $E^3$. Similarly, $E^3$ is proximate to $E^4$ and $E^6$, $E^4$ is proximate to $E^5$ and $E^6$, 
and $E^5$ is proximate to $E^6$ only.

This means that we start with $D^1 = E^1$. Then, since only $E^1$ is proximate to $E^2$, we have $D^2 = D^1 + E^2 = E^1 + E^2$.
The divisors proximate to $E^3$ are $E^2$ and $E^1$, thus $D^3 = D^2 + D^1 + E^3 = 2 E^1 + E^2 + E^3$.
Divisors $E^3$ and $E^2$ are proximate to $E^4$, so $D^4 = D^3 + D^2 + E^4 = 3 E^1 + 2 E^2 + E^3 + E^4$.
Only $E^4$ is proximate to $E^5$, so $D^5 = D^4 + E^5 = 3 E^1 + 2 E^2 + E^3 + E^4 + E^5$.
Divisors $E^5, E^4$ and $E^3$ are all proximate to $E^6$:
$D^6 = D^5 + D^4 + D^3 + E^6 = 8 E^1 + 5 E^2 + 3 E^3 + 2 E^4 + E^5 + E^6$. We will return to this example in the proof of Lemma~\ref{lem:component}.
\end{example}

We will now use the blowup formula in Heegaard Floer homology to identify the $Spin^c$ structures on $X$ with  $F^+_{W(\Gamma), \s}(c^+(\xi_0)) \neq 0$. By contrast with 
Lemma~\ref{lem:canhom-min},  when $\Gamma$ has vertices of weight $-1$ it is no longer true that $F^+_{W(\Gamma), \s}(c^+(\xi_0))=0$ for all 
$Spin^c$-structures $\s$ other than the canonical one.  Recall that $K_0 = c_1(TX, J)$, and let $\s_{min, can}$ denote the canonical $Spin^c$ structire on $X_{min}$.

\begin{lemma} \label{lem:blowups}  Let $\s$ be a $Spin^c$-structure on $X$.

(1)  $F^{+}_{W, \s} (c^+(\xi_0))$ is the generator of $\T_0^+$ in degree $0$ if $\langle c_1(\s), D \rangle  = \pm 1$ for all 
$D \in \mathcal{D}$, and $\s|X_{min}= \s_{min, can}$.  Otherwise, $F^{+}_{W, \s} (c^+(\xi_0))=0$.

(2) Equivalently, $F^{+}_{W, \s} (c^+(\xi_0)) = 1 \in \T_0^+$
 if
$c_1(\s)= K_0 + 2 \PD(D)$ for some  $D \in \mathcal{S}$; otherwise,   $F^{+}_{W, \s} (c^+(\xi_0))=0$.
\end{lemma}

\begin{cor}\label{cor:phipsi} (1) The element $c=T^+(c^+(\xi_0))\in \HH_0^+(\Gamma)$  is given by the function $\phi_0 \in \HH_0^+(\Gamma)$ such that
$$
\phi_0(\s)=  \begin{cases}  1 \in \T_0^+  \textrm{\ if\  } \s = K_0 + 2 \PD(D), D\in \mathcal{S}, \\
0 \textrm{\ otherwise}.
\end{cases}
$$

(2) Under the isomorphism $i_{K_0}^*$ of Lemma~\ref{lem:charvsl}, the function
$\phi_0$ corresponds to the  function $\varphi_0 \in \HH\LL_0^+(\Gamma, K_0)$
such that for $x\in L$, 
$$
\varphi_0(x) =  \begin{cases}  1 \in \T_0^+  \textrm{\ if\  } x \in \mathcal{S}, \\
0 \textrm{\ otherwise}.
\end{cases}
$$
\end{cor}

\begin{proof}[Proof of Corollary~\ref{cor:phipsi}] This follows from Lemma~\ref{lem:blowups}, by an argument completely 
 analogous to the proof of Lemma~\ref{lem:canhom-min} \end{proof}.

\begin{proof}[Proof of Lemma~\ref{lem:blowups}]  

(1) In the minimal case, the set $\mathcal{S}$ is empty, and the statement was already given in Lemma~\ref{lem:canhom-min}.

In the non-minimal case, we blow down $X$ to $X_{min}$, perhaps in several steps, and argue by induction on the 
number of exceptional divisors in this sequence of blowdowns.

Suppose the statement holds for the complex surface $X'$ which is a blowup of $X_{min}$, and $X$ is obtained from 
$X'$ by an additional blowup  with exceptional divisor $D_0$.
Let $\mathcal D'$ and $\mathcal S'$ be the sets of homology classes defined for $X'$ as in formulas~(\ref{Dclasses}) and~(\ref{Sclasses}). 
Note that $H_2(X)=H_2(X')\oplus \ZZ$;  identifying homology 
classes in $X'$ with their images under inclusion in $X$, we have $\mathcal D = \mathcal D' \cup \{D_0 \}$.
Let $W$, $W'$ denote the corresponding cobordisms from  $-Y^3$ to $S^3$, and let $B= [0, 1] \times Y \# \overline{\CP^2}$ be the blowup cobordism.
As a smooth manifold, $W$ can be thought of as composition of cobordism $B$ from $-Y$ to $-Y$ followed by cobordism $W'$ from $-Y$ to $S^3$.
Since $W= W' \#\overline{\CP^2}$,   a $Spin^c$-structure $\s$ on $W$ is completely determined  
by its restrictions $\s|_{W'}$  and $\s|_B$, so by the composition law \cite[Theorem 3.4]{OSfour}, 
we have 
$$
F^+_{W', \s|_{W'}} \circ  F^{+}_{B, \s|B} = F^{+}_{W, \s}.
$$
We are interested in  the canonical contact structure $\xi_0$ on 
$Y = \partial X' = \partial X$, so we focus on  $Spin^c$-structures with $\s|_Y =\s_{can}|_{Y}$, since $c^+(\xi_0) \in HF^+(Y, \s_{can}|_{Y})$.

The map $F^{+}_{B, \s|B}$ is given by the Ozsv\'ath--Szab\'o
blowup formula \cite[Theorem 3.7]{OSfour}. Let $D_0$ be the exceptional divisor of the blowup, and 
$\langle c_1(\s), D_0 \rangle =  \langle c_1(\s|B), D_0 \rangle =  \pm (2l +1)$,  $l \geq 0$.
Then 
$$
F^{+}_{B, \s|B}: HF^{+}\left ({-Y, \s|_Y} \right) \to  HF^{+}\left( {-Y, \s|_Y} \right) 
$$ is the multiplication by $U^{l(l+1)}$. 
By~\cite{OScont}, the contact invariant lies in $\Ker U$, 
thus we get   we get $F^{+}_{B, \s|B} (c^+(\xi_0))=0$  if $l \neq 0$. For $l=0$,  
$F^{+}_{B, \s|B} (c^+(\xi_0))= c^+(\xi_0)$
since $F^{+}_{B, \s|B}$ is the identity map.

By the induction hypothesis, $F^{+}_{W', \s|_{W'}} (c^+(\xi_0))$ is the generator  $1 \in \T_0^+$  if   $\s|_{W_{min}} = \s_{min, can}$ and 
$  \langle c_1(\s), D' \rangle = \pm 1$ for all $D' \in \mathcal D'$; moreover,  $F^{+}_{W', \s|_{W'}} (c^+(\xi_0))$ is zero for all other $Spin^c$ structures.
Then from the composition formula,   
$F^{+}_{W, \s|_{W}} (c^+(\xi_0)) =1 \in \T_0^+$  if 
\begin{equation*}
 \s|_{W_{min}}= \s_{min, can}, \quad \langle c_1(\s), D' \rangle = \pm 1 \text{ for all } D' \in \mathcal D', \text{ and }\langle c_1(\s), D_0 \rangle = \pm 1. 
 \end{equation*}
For all other $Spin^c$ structures on $X$, $F^{+}_{W, \s} (c^+(\xi_0))=0$. 
Since $\mathcal D = \mathcal{D'} \cup \{D_0 \}$,  part (1) of the lemma holds for $Spin^c$ structures on  $X$. 

(2) Observe that  
for a $Spin^c$-structure $s$ with  $\s|_{W_{min}}  = \s_{can}$ and $\langle c_1(\s), D \rangle = \pm 1$ for all $D \in \mathcal D$, we have 
$$
c_1(\s)= c_1(\s_{can, W_{min}}) + \sum_{m,j} \pm \PD(D^m_j).    
$$
The canonical  $Spin^c$ structure 
$\s_{can}$ on $W$ is determined by  the equalities  $\s|_{W_{min}}  = \s_{can, W_{min}}$ and $\langle c_1(\s), D^m_j \rangle = + 1$ for all $D^m_j \in \mathcal{D}$, so  
$$
K_0=c_1(\s_{can})= c_1(\s_{can, W_{min}}) - \sum_{m,j} \PD(D^m_j).   
$$
The  statement now follows from the first part of the lemma.
\end{proof}

We now show that part (3) of Lemma~\ref{lem:canhom-min}  holds in the general case even though parts (1) and (2) don't. 
Corollary~\ref{cor:phipsi} identifies the zero set of
the function $\varphi_0$ as the set $\mathcal S \subset H_2(X;\ZZ)$ given by sums of distinct homology classes  
from $\mathcal{D}$. To understand the function $\psi_0$ corresponding to $\varphi_0$ in  the graded root description of lattice homology, we examine 
connected components of the level set $\overline{L}_{K_0, \leq 0}$ (compare with Remark~\ref{rem:C=0}).

\begin{lemma}\label{lem:component} (1) Let $C_0$ be the connected component of $0$ in the level set $\overline{L}_{K_0, \leq 0}$.
Then $\mathcal{S}$ is contained in $C_0$.  
(In fact, $\mathcal S =C_0$.) 
 
 (2)  Under the isomorphism $\theta^*$ of Lemma~\ref{lem:latticevsgradedroot}, 
the function $\varphi_0$ corresponds to the function $\psi_0: \V(R_{K_0}) \to
\T_0^+$
such that $\psi_0(w_0)$ is the generator of $\T_0^+$ in degree $0$, and
$\psi_0=0$
for all other vertices of $R_{\ts_{can}}$. The vertex $w_0$ corresponds to the connected component of $0$ in the homology lattice.
\end{lemma}

\begin{proof} (1) We saw that  the homology classes  $D^m_j$ in the set $\mathcal D$
can be represented by pairwise disjoint symplectic spheres of self-intersection $(-1)$ in $X = X_0$. 
 It follows that these classes all lie in the zero level set of $\chi_{K_0}$: by  adjunction formula,
$\langle K_0, D^m_j \rangle = \langle c_1(TX, J), D^m_j \rangle =   D^m_j \cdot D^m_j  + 2 = 1$, therefore,
\[ \chi_{K_0}(D^m_j) = 0. \]  Clearly,  $  D^m_j \cdot D^n_i =0$ for any two distinct classes  since they have disjoint representatives.  
The property of the weight function 
$\chi_{K_0}(x+y) = \chi_{K_0}(x) + \chi_{K_0}(y) - x \cdot y $ then implies that $\chi_{K_0} = 0$ for any of the sums forming $\mathcal{S}$, thus 
$\mathcal{S}$ lies in the zero level set $\overline{L}_{K_0, \leq 0}$. It will be convenient to refer to   
$D^m_1, D^m_2, \dots, D^m_{k(m)}$ as elements of depth $m$, $1\leq m\leq R$.

We want to show that $\mathcal{S}$ lies in the connected component $C_0$. 
Recall from section~\ref{ss:lattice} that 1-cells in $\overline{L}_{K_0, \leq 0}$ correspond to
basis vectors in the lattice $L=H_2(X, \ZZ)$: we can walk from $x\in L$ to 
$x+v \in L$ along the edge corresponding to $v$ if $v$ is a basic vector. For each element  of depth 1, the homology class 
$D^1_i=E^1_i$ is in the basis, thus $D^1_i=0+E^1_i$  lies in $C_0$ because it is connected to 0 by the egde corresponding to $E^1_i$. 
A typical lattice point $F\in \mathcal{S}$ is a sum of some elements $D^m_j$. We will show that $F \in C_0$ recursively, by reducing to sums of elements 
of lower depth. The key idea is Formula \eqref{eq:prox}, which recursively expresses 
classes $D^m_j$ in terms of the basis elements $E^1_1, \dots, E^R_{k(R)}$. Notice that in the above formula, 
each instance of $n$ on the right hand side is strictly less than $m$, thus $D^m_j = E^m_j + \text{sum of elements of lower depth}$.


First, we illustrate the recursive procedure for the lattice from  Example~\ref{ex:div}. 
In fact, we construct the paths from 0 starting with elements of lower depth and building up to higher depth;
the recursion will work backwards, reducing the depth. 
To begin, observe that  $D^1 \in C_0$ as above. Next, $D^2=D^1+E^2$ is one step away from $D^1$ along the edge 
$E^2$, so $D^2 \in C_0$. Further, $D^1+D^2= D^2+E^1$ is connected to $D^2$ by the edge $E^1$, so $D^1+D^2 \in C_0$. 
Next, we see that $D^3=D^1+D^2+E^3$ is connected  
to the vertex $D^1+D^2 \in C_0$ by the edge $E^3$, so $D^3 \in C_0$. Then, we establish that $D^1+D^3= D^3+E^1 \in C_0$, and then  $D^2+D^3=D^3+D^1+E^1 \in C_0$, 
and then  $D^1+D^2+D^3= D^2+D^3+E^1 \in C_0$.
Now we see that  $D^4= D^2+D^3 + E^4 \in C_0$ because we already know that $D^2+D^3 \in C_0$. We can write out the full path 
from 0 to $D^4$:
$$
D^4=\underbrace{\underbrace{\underbrace{E^1}_{D^1} + E^2}_{D^2} + \underbrace{E^1}_{D^1} + E^3}_{D^3}
 + \underbrace{\underbrace{E^1}_{D^1} + E^2}_{D^2} + E^4. 
$$
In this form, 
$D^4 = \sum_{j = 1}^k E^{i(j)}$ (with possibly repeating indices $i(j)$) is the sum of basis elements, showing steps along the corresponding edges. 
The initial partial sums $P^m = \sum_{j = 1}^m E^{i(j)}$ we obtain after each step correspond to the vertices of the lattice $L$ that our path goes through, 
starting from 0. The underbraces show that each such vertex 
is a sum of some \emph{distinct} elements $D^j$. 
All of these sums lie in $\mathcal{S}$ which is contained the level set $\overline{L}_{K_0, \leq 0}$, thus we see
that all the lattice points on the path lie in $\overline{L}_{K_0, \leq 0}$, showing that the entire path lies in the connected component of 0.

We continue in this fashion, consecutively building paths to elements $D^4+D^1$, $D^4+D^2$, $D^4+D^2+D^1$, $D^4+D^3$, $D^4+D^3+D^1$ etc, until we reach all elements 
in $\mathcal{S}$. Here is the
recursively constructed full path showing that $D^6$ is in the connected component of $0$:
\[ 
\underbrace{\underbrace{
\underbrace{\underbrace{\underbrace{E^1}_{D^1} + E^2}_{D^2} + \underbrace{E^1}_{D^1} + E^3}_{D^3}
 + \underbrace{\underbrace{E^1}_{D^1} + E^2}_{D^2} + E^4}_{D^4} + E^5}_{D^5} + 
 \underbrace{
\underbrace{\underbrace{\underbrace{E^1}_{D^1} + E^2}_{D^2} + \underbrace{E^1}_{D^1} + E^3}_{D^3}
 + \underbrace{\underbrace{E^1}_{D^1} + E^2}_{D^2} + E^4}_{D^4} + \]
 \[
 + \underbrace{\underbrace{\underbrace{E^1}_{D^1} + E^2}_{D^2} + \underbrace{E^1}_{D^1} + E^3}_{D^3} + E^6 = D^6.
\]

We now return to the proof of  Lemma~\ref{lem:component} and proceed with the formal recursion.                                     
   To every lattice point $F \in \mathcal{S}$ we can assign an $R$-tuple of integers  \[ t(F) := (s_R, \dots, s_2, s_1), \] where $s_r$ is the number of summands 
in $F$ that are elements of depth $r$.  Namely, if  
$F = \sum_{a,i} \varepsilon^a_i D^a_i$, $\varepsilon^a_i \in \{0,1\}, D^a_i \in \mathcal{D}$, we set $s_r := \sum_{i = 1}^{k(r)} \varepsilon^r_i$, $1\leq r\leq R$.

We define a partial ordering on $\mathcal{S}$ as follows. 
For $F_1, F_2 \in \mathcal{S}$ we say that $F_1 \prec F_2$ if $t(F_1) < t(F_2)$ with respect to the lexicographical ordering, 
looking for the first differing number from the left (the largest index $r$ where there is no equality). In other words, $F_1 \prec F_2$ if $F_1$ has fewer elements 
of depth $r$ than $F_2$, and $F_1$ and $F_2$ have the same number of elements of depth $r+1, r+2, \dots, R$.   

Take any element $F \in \mathcal{S}$.  Let $u$ be the smallest upper index among the summands of $F$, and write
$F$ in the form $F = \overline{F} + D^u_i$,
where all summands in  $\overline{F}$ have upper index at least $u$  ($\overline{F}$ might be an empty sum). 
Looking at Formula~\eqref{eq:prox}, we get $D^u_i = \overline{D} + E^u_i$, where all summands $\overline{D} \in \mathcal{F}_{u-1}$  have 
 upper index less than $u$. Again, $\overline{D}$ might be an empty sum.

It follows that $\overline{F}$ and $\overline{D}$ cannot share any summands, so the sum $\widetilde{F} = \overline{F} + \overline{D}$ lies in $\mathcal{S}$.
Moreover, due to our choice of the index $u$, we have $\widetilde{F} \prec F$. Because $F = \widetilde{F} + E^u_i$, this shows that 
any nontrivial element of $\mathcal{S}$ can be written as a sum of a lexicographically 
lower element and a basis element. To illustrate,  we would decompose the element $F=D^4+D^3$  in the above example as $F = \widetilde{F} +E^3$, where  
$\widetilde{F}=D^4+D^2+D^1$ is lexicographically lower than $F$; similarly, $D^4+D^2+D^1= D^4+D^2+E^1$, with  $D^4+D^2 \prec D^4+D^2+D^1$. 

The relation $F = \widetilde{F} + E^u_i$ implies that $F$ is connected to $\widetilde{F}$ by the edge $E^u_i$. Therefore,
if $\widetilde{F}$ is in the connected component $C_0$, then so is $F$. Since the empty sum equals 0 and obviously lies in $C_0$, the above recursion 
shows that  every element of $\mathcal{S}$ is in $C_0$.

It is worth noting that the above decomposition $F = \widetilde{F} + E^u_i$ provides a recursive
way to write any sum from $\mathcal{S}$ as a series of basis elements,
indicating the path starting at zero and ending at the given lattice point, as shown in the  example above.

We have shown that $\mathcal{S} \subset C_0$. Corollary~\ref{cor:phipsi}   identifies $\mathcal{S}$ with the zero set of $\varphi_0$. 
The proof of Lemma~\ref{lem:latticevsgradedroot} implies that 
that the zero set of the function $\varphi_0$ is the union of several connected components of  $\overline{L}_{K_0, \leq 0}$ (in particular, 
$\varphi_0$ must vanish on {\em entire} connected components).       
It follows that $\mathcal{S} = C_0$. 

(2) The second statement of the lemma follows from the first statement, Corollary~\ref{cor:phipsi}, and Lemma~\ref{lem:latticevsgradedroot}. (Compare with Part (3) of 
Lemma~\ref{lem:canhom-min}.)
\end{proof}

\begin{remark}\label{Zcoeffs} In this paper we worked with coefficients in
$\FF=\ZZ/2$ for simplicity, however our results hold for integer coefficients as
well. When working with $\ZZ$ coefficients, the contact invariant $c^+(\xi)$ is
only defined up to sign. The results of \cite{Pla, Ghi} then assert that 
$F^+_{W(\Gamma), \s_{can}}(c^+(\xi))$ is a generator
 $\pm 1 \in \T_0^+$, where  $\T_0^+$ now stands for  $\ZZ[U, U^{-1}]/U \cdot
\ZZ[U]$. A further issue is that cobordism maps are only defined up to sign in
\cite{OSfour}, although \cite[\textsection 2.1]{OSplumb} explains how to define
the map $T^+$ up to one overall sign (which can also be fixed). The isomorphisms
between various constructions of lattice cohomology work with $\ZZ$
coefficients, and the distinguished elements $\phi_0$, $\varphi_0$, and $\psi_0$
of Lemma~\ref{lem:canhom-min} and its analogs in Section~\ref{s:blowups} correspond to one another. Thus, we see that 
the element $c^+(\xi_0) \in HF^+(-Y(\Gamma), \ts)/\pm 1$ is mapped to $(\pm
\psi_0)$ under the map 
$ HF^+(-Y(\Gamma), \ts)/\pm 1 \to  \HH (R_{K_0}, \chi_{K_0})/\pm 1$, and our
proof goes through as before. 
\end{remark}



 

\end{document}